\documentclass[reqno]{amsart}

\usepackage{pgfplots}
\usepackage{amsthm}
\usepackage{latexsym}
\usepackage{dsfont}
\usepackage{bbm}
\usepackage{amssymb}
\usepackage{amsmath}
\usepackage{mathtools}
\usepackage{subcaption}
\usepackage{tikz}
\usepackage{color}

\usepackage{enumerate}

\newcommand{\defeq}{\vcentcolon=}
\newcommand{\eqdef}{=\vcentcolon}
\newcommand{\Law}{\mathcal{L}}

\newcommand{\weaklyto}{%
  \mathrel{\vbox{\offinterlineskip\ialign{%
    \hfil##\hfil\cr
    $\scriptscriptstyle\mathrm{w}$\cr
    $\to$\cr
}}}}
\newcommand{\distto}{%
  \mathrel{\vbox{\offinterlineskip\ialign{%
    \hfil##\hfil\cr
    $\scriptscriptstyle\mathrm{d}$\cr
    \noalign{\kern-.05ex}
    $\to$\cr
}}}}
\newcommand{\Probto}{%
  \mathrel{\vbox{\offinterlineskip\ialign{%
    \hfil##\hfil\cr
    $\scriptscriptstyle\Prob$\cr
    \noalign{\kern-.05ex}
    $\to$\cr
}}}}

\newcommand{\N}{\mathbb{N}}
\newcommand{\Z}{\mathbb{Z}}

\newcommand{\R}{\mathbb{R}}
\newcommand{\B}{\mathcal{B}}
\newcommand{\C}{\mathbb{C}}

\newcommand{\I}{\mathcal{I}}

\newcommand{\imag}{\mathrm{i}}

\newcommand{\Prob}{\mathbb{P}}

\newcommand{\E}{\mathbb{E}}

\newcommand{\F}{\mathcal{F}}

\newcommand{\cZ}{\mathcal{Z}}
\newcommand{\1}{\mathds{1}}

\newcommand{\meas}{\mathcal{M}}

\newcommand{\eqdist}{%
  \mathrel{\vbox{\offinterlineskip\ialign{%
    \hfil##\hfil\cr
    $\scriptscriptstyle\mathrm{law}$\cr
    \noalign{\kern.2ex}
    $=$\cr
}}}}

\newcommand{\ds}{\mathrm{d} \mathit{s}}
\newcommand{\dt}{\mathrm{d} \mathit{t}}
\newcommand{\du}{\mathrm{d} \mathit{u}}

\newcommand{\dx}{\mathrm{d} \mathit{x}}
\newcommand{\dy}{\mathrm{d} \mathit{y}}

\theoremstyle{plain}
\newtheorem{theorem}{Theorem}[section]

\newtheorem{lemma}[theorem]{Lemma}
\newtheorem{proposition}[theorem]{Proposition}
\theoremstyle{remark}

\theoremstyle{definition}
\newtheorem{definition}[theorem]{Definition}

\numberwithin{equation}{section}

\begin{document}

\title[Self-similar solutions to kinetic-type evolution equations]{Self-similar solutions to kinetic-type evolution equations: beyond the boundary case}

\authors{Dariusz~Buraczewski, Konrad~Kolesko and Matthias~Meiners}
\date{\today}

\address{Dariusz Buraczewski, Mathematical Institute, University of
Wroc{\l}aw, Plac Grunwaldzki 2/4, 50-384 Wroc{\l}aw, Poland}
	\email{dariusz.buraczewskid@uwr.edu.pl}

\address{Konrad~Kolesko, Institut f\"ur Mathematik, Universit\"at Innsbruck, Technikerstr.~13, 6020 Innsbruck, Austria,
and Mathematical Institute, University of
Wroc{\l}aw, Plac Grunwaldzki 2/4, 50-384 Wroc{\l}aw, Poland}
	\email{kolesko@math.uni.wroc.pl}

\address{Matthias~Meiners, Institut f\"ur Mathematik, Universit\"at Innsbruck, Technikerstr.~13, 6020 Innsbruck, Austria}
	\email{matthias.meiners@uibk.ac.at}


\begin{abstract}
We study the asymptotic behavior as $t \to \infty$ of a time-dependent family $(\mu_t)_{t \geq 0}$
of probability measures on $\R$ solving the kinetic-type evolution equation $\partial_t \mu_t + \mu_t = Q(\mu_t)$
where $Q$ is a smoothing transformation on~$\R$.
This problem has been investigated earlier,
e.g.\ by Bassetti and Ladelli [\emph{Ann.~Appl.~Probab.} 22(5): 1928--1961, 2012]
and Bogus, Buraczewski and Marynych [\emph{Stochastic~Process.~Appl.} 130(2):677--693, 2020].
Combining the refined analysis of the latter paper,
which provides a probabilistic description of the solution $\mu_t$
as the law of a suitable random sum related to a continuous-time branching random walk at time $t$,
with recent advances in the analysis of the extremal positions in the branching random walk
we are able to solve the remaining case that has been left open until now. In the course of our work,
we significantly weaken the assumptions in the literature that guarantee the existence (and uniqueness)
of a solution to the evolution equation $\partial_t \mu_t + \mu_t = Q(\mu_t)$.
\end{abstract}

\maketitle

\section{Introduction}		\label{sec:Intro}

Given a sequence $A = (A_1,A_2,\ldots)$ of non-negative random variables
with $N \defeq \max\{j: A_j \not = 0\} < \infty$ almost surely
we consider the kinetic-type evolution equation
\begin{equation}	\label{eq:kinetic-type equation}
\partial_t \mu_t + \mu_t = Q(\mu_t)
\end{equation}
for a time-dependent family $(\mu_t)_{t \geq 0}$ of probability measures on $\R$
equipped with the Borel $\sigma$-algebra $\B(\R)$
where \eqref{eq:kinetic-type equation} has to be understood in the weak sense
and $Q$ is the smoothing transformation associated with $A$.
More precisely,
the smoothing transformation $Q$ is a self-map of $\meas^1(\R)$,
the set of probability measures on $(\R,\B(\R))$,
and is defined by the formula
\begin{equation*}	\label{eq:Q}
Q(\mu) =  \Law\bigg(\sum_{j=1}^N A_j X_j\bigg),
\end{equation*}
where $\Law(Y)$ denotes the law of a random variable $Y$
and  $X_1,X_2,\ldots$ are i.i.d.\ and independent of $A$ with $X_j \sim \mu$, $j \in \N$.
On the level of the Fourier transform,
\eqref{eq:kinetic-type equation} corresponds to the Cauchy problem
\begin{equation}	\label{eq:kinetic-type equation FT}
\partial_t \phi_t(\xi)+\phi_t(\xi)=\widehat{Q}(\phi_t)(\xi),	\quad t\ge 0,\ \xi\in\R
\end{equation}
where
the boundary condition $\phi_0$ is the Fourier transform of a given $\mu_0 \in \meas^1(\R)$
and $\widehat{Q}$ is a self-map of the set of characteristic functions of probability measures on $(\R,\B(\R))$ defined by
\begin{equation}\label{eq:Q_hat}
\widehat{Q}(\phi)(\xi) \defeq \E\bigg[\prod_{j=1}^N \phi(A_j \xi)\bigg], \quad \xi \in \R,
\end{equation}
for $\phi$ being  the Fourier transform of some probability measure $\mu \in \meas^1(\R)$.

Under suitable assumptions (see e.g.\ Theorem \ref{Thm:existence and uniqueness} below
or \cite[Proposition 2.5]{Bogus+al:2019}),
given an initial law $\mu_0$, Eq.~\eqref{eq:kinetic-type equation} has a unique solution,
which we shall denote by $(\mu_t)_{t \geq 0}$ henceforth.
The corresponding family of Fourier transforms will be denoted by $(\phi_t)_{t \geq 0}$.
The behavior of the solution to \eqref{eq:kinetic-type equation} is strongly related with the \emph{spectral function}
$F(\theta)\defeq\Phi(\theta)/\theta, $ where
\begin{equation}	\label{eq:Phi(theta)}
\Phi(\theta) \defeq \E \bigg[\sum_{j=1}^N A_j^\theta \bigg] - 1,	\quad	\theta \geq 0.
\end{equation}

\subsection{Motivation and related models in the literature}
Let us now briefly present some models that fit into the framework of Equation \eqref{eq:kinetic-type equation}.
Most of the models have a fixed number of $A_j \not = 0$, i.e., $A=(A_1,\dots,A_N)$ with constant $N$.

The case $N=2$ and $A = (\sin U,\cos U)$,
$U$ being uniformly distributed on $[0,2\pi)$,
was considered by Kac \cite{Kac:1956} as a model for the behavior of a particle in a homogeneous gas,
where particles collide at random times.
It is known as the $1$-dimensional Kac caricature.
The distribution $\mu_t$ represents the law of the velocity of a randomly chosen particle
and the operator $Q$ describes the change of velocity after collision of two particles.\footnote{
Although $A_1,A_2$ are not nonnegative in the $1$-dimensional Kac caricature,
the model can be rephrased in the above setup
as $(\sin U, \cos U)$ has the same law as $(\epsilon_1 |\sin U|, \epsilon_2 |\cos U|)$
with independent $\epsilon_1,\epsilon_2$ that are uniform on $\{-1,1\}$.
Then one can replace $(\sin U, \cos U)$ by $(|\sin U|, |\cos U|)$
and replace the $X_j$ in the definition of $Q$ by $\epsilon_j X_j$, $j=1,2$
where $(\epsilon_1, \epsilon_2)$ and $(X_1,X_2)$ are independent,
which corresponds to restricting the smoothing transform to symmetric laws on $\R$.
}
In subsequent works, the model was extended in various directions,
for instance to non-conservative kinetic models, see e.g.~\cite{Pareschi+Toscani:2006}.

The kinetic evolution equation \eqref{eq:kinetic-type equation} also
found applications in models for wealth redistribution in econophysics.
Loosely speaking, gas particles become agents and the velocity of a particles becomes the agent's wealth.
More precisely, we consider a class of models with indistinguishable agents.
The agent state is characterized by his current wealth $w \geq 0$.
The interaction between two agents is described by
\begin{align*}
v^*&=p_1v+q_1w,
\quad
w^*=q_2v+p_2w,
\end{align*}
where $(v,w)$ and $(v^*,w^*)$ stand for the pre- and post-trade wealths of the two agents, respectively.
The coefficients $p_i$ and $q_i$ are assumed to be random representing the risk of the market.
The idea with random coefficients is due to the fact that agents may invest some of their money in risky assets.
It is common to  assume that the society's mean wealth is preserved on average, i.e.,
$\E[p_1+q_1+q_2+p_2]=2$.
In our framework it can be represented by choosing
\begin{equation}
\label{eq:econophysics model}
A=(A_1,A_2) \defeq (\epsilon p_1+(1-\epsilon)q_2,\epsilon q_1+(1-\epsilon)p_2),
\end{equation}
where $\epsilon$ is an independent Bernoulli variable with success parameter $\frac12$.
The conservation of mean translates to $\Phi(1)=0$.
It has been shown \cite{Duering+al:2008,Matthes+Toscani:2008} that if $\Phi(r)<0$ for some $r>1$ and the expectation of $\mu_0$ is finite,
then $\mu_t$ converges to some steady state $\mu_\infty$ which has either a Pareto tail or a slim tail.
On the other hand, if $\Phi(r) > 0$ for all $r>1$,
then $\mu_t \weaklyto \delta_0$, in other words, a typical agent goes bankrupt.
Therefore, it is natural to investigate the rate of decay of the wealth of a typical agent as $t \to \infty$.

We refer the reader to
\cite{Bassetti+Ladelli:2012,Bassetti+al:2011,Bassetti+al:2015,Bassetti+Perversi:2013, Bobylev+al:2009}
for examples and a more comprehensive account to the literature.

\subsection{State of the art and assumptions}	\label{subsec:assumptions}

The following assumptions concerning $A$ will be relevant in the paper:
\begin{enumerate}[\bf{(A}1)]
	\item
		$\Prob(A_1, A_2 \ldots \in ar^\Z \cup \{0\}) < 1$ for all $r >1$ and $1 \leq a < r$;		
\label{enum:non-lattice}
	\item
		There is $\vartheta>0$ such that $\Phi(\vartheta)  < \infty$,
		\begin{align}
   		\vartheta\E \bigg[\sum_{j \geq 1}A_j^\vartheta  \log A_j\bigg]+1
		&= \E \bigg[\sum_{j \geq 1} A_j^\vartheta \bigg]	\label{eq:boundary case assumption}	\\
\text{and}	\qquad
   		\E \bigg[\sum_{j \geq 1} A_j^\vartheta  \log^2 A_j\bigg] &< \infty.		\label{eq:2nd moment assoc RW}
		\end{align}	\label{enum:existence vartetha}
	\item
		For $X\defeq \sum_{j \geq 1} A_j^\vartheta$ and $\tilde X \defeq \sum_{j \geq 1} A_j^\vartheta\log_+A_j$
		it holds that
		\begin{equation*}
		\E[X\log_+^2X]<\infty\text{ and }\E[\tilde X\log_+ \tilde X]<\infty,
		\end{equation*}
		where $x_\pm\defeq \max(\pm x,0)$ for $x \in \R$.	\label{enum:XlogX condition}
    \item  For any $0<\delta<1$
    		\begin{equation*}
		\int_{1-\delta}^1\frac{\ds}{|\E[s^N]-s|}=\infty.
		\end{equation*}
        \label{enum:finitness of BRW}
\end{enumerate}
Assumption (A\ref{enum:non-lattice}) is a non-lattice assumption, while (A\ref{enum:finitness of BRW})
guarantees non-explosion of a related Markov branching process, see the discussion below \eqref{eq:m(t,theta)}.
Notice that $\E[N]<\infty$ is sufficient for (A\ref{enum:finitness of BRW}).
%

Notice that if $\Phi(\theta) < \infty$, then $F(\theta)$ equals the tangent of the angle between
the line segment joining $(0,0)$ and $(\theta,\Phi(\theta))$ and the positive horizontal half-axis.
If $\Phi$ is defined on some open neighborhood of $\vartheta$,
then the relation \eqref{eq:boundary case assumption} states that
$\Phi'(\vartheta) = F(\vartheta)$, i.e., $\vartheta$ is the unique minimizer of $F$.

The asymptotic behavior of $\mu_t$ as $t\to\infty$
depends on the interplay between the minimizer $\vartheta$ and the initial condition $\mu_0$.
More precisely, it depends on the relation between $\vartheta$ and $\gamma\in(0,2]$,
where $\gamma$ is such that $\phi_0(\xi) \sim 1-c_\pm|\xi|^\gamma$ as $\xi\to0^\pm$,
i.e., $\mu_0$ is in the domain of normal attraction of a $\gamma$-stable
law (and if $\mu_0$ is additionally centered when $\gamma > 1$).

The vast majority papers are treating the case where $\gamma<\vartheta$. In this case
\begin{equation}
\label{eq:subcritical asymptotic}
\phi_t(e^{-F(\gamma)t}\xi)\to \phi_{\infty}(\xi)
\end{equation}
where $\phi_\infty$ is the characteristic function of a non-degenerate probability distribution on $\R$
(cf.~\cite{Bobylev+al:2009} for an analytical approach and \cite{Bassetti+Ladelli:2012} for a probabilistic interpretation).

In the recent work \cite{Bogus+al:2019}
the authors establish a connection with continuous-time branching processes
which enables them to treat \emph{the boundary case} $\gamma=\vartheta$ in which
\begin{equation}
\label{eq:critical asymptotic}
\phi_t(t^{\frac1{2\vartheta }}e^{-F(\vartheta)t}\xi)\to \phi_{\infty}(\xi)
\end{equation}
again for the characteristic function $\phi_\infty$ of a non-degenerate probability measure on $\R$.

The purpose of this paper is to  fill the gap in the theory of one-dimensional kinetic-type equations
by treating the remaining case $\vartheta<\gamma$.
We demonstrate how the asymptotic behavior of $\mu_t$
can be derived from recent progress on kinetic-type equations \cite{Bogus+al:2019}
and on the extrema of branching random walks
\cite{Iksanov+Kolesko+Meiners:2018,Madaule:2017}.
Our proof works under a mild $X \log X$-type moment condition
(cf.~assumption (A\ref{enum:XlogX condition}))
and for random $N$.
We mention that the assumptions in the earlier results concerning the cases $\gamma<\vartheta$ or $\gamma=\vartheta$ may be weakened analogously.

\begin{figure}[h]
\begin{subfigure}[t]{0.3\textwidth}
\begin{center}
\begin{tikzpicture}[scale=0.6]
		\draw [help lines] (0,-1) grid (5,2);
		\draw [thick,->] (0,-1.05) -- (0,2.05);
		\draw [thick,->] (-0.05,0) -- (5.05,0);
		\draw (5,0) node[right]{$\theta$};
		\draw (0,2) node[above]{$\Phi(\theta)$};
		\filldraw[color=black] (1.3,-0.14) circle(0.4ex);
		\draw [densely dotted,-] (1.3,-0.14) -- (1.3,0) node[above] {\small $\gamma$};
		\filldraw[color=black] (2.061553,-0.4519411) circle(0.4ex);
		\draw [densely dotted,-] (2.061553,-0.4519411) -- (2.061553,0) node[above] {\small $\vartheta$};	
		\draw [-] (0,0) -- (3.9,-0.42);
		\draw [thick, domain=0:5, samples=128, color=black] plot(\x, {(\x/2-1.25)^2-0.5});
\end{tikzpicture}
\end{center}
\subcaption{\small The case covered by Bassetti and Ladelli \cite{Bassetti+Ladelli:2012}.}
\end{subfigure}
\hfill
\begin{subfigure}[t]{0.3\textwidth}
\begin{center}
\begin{tikzpicture}[scale=0.6]
		\draw [help lines] (0,-1) grid (5,2);
		\draw [thick,->] (0,-1.05) -- (0,2.05);
		\draw [thick,->] (-0.05,0) -- (5.05,0);
		\draw (5,0) node[right]{$\theta$};
		\draw (0,2) node[above]{$\Phi(\theta)$};
		\filldraw[color=black] (2.061553,-0.4519411) circle(0.4ex);
		\draw [densely dotted,-] (2.061553,-0.4519411) -- (2.061553,0) node[above] {\small $\vartheta$};		
		\draw [-] (0,0) -- (4.123106,-0.9038822);
		\draw [thick, domain=0:5, samples=128, color=black] plot(\x, {(\x/2-1.25)^2-0.5});
\end{tikzpicture}
\end{center}
\subcaption{\small The boundary case covered by Bogus et al.~\cite{Bogus+al:2019}.}
\end{subfigure}
\hfill
\begin{subfigure}[t]{0.3\textwidth}
\begin{center}
\begin{tikzpicture}[scale=0.6]
		\draw [help lines] (0,-1) grid (5,2);
		\draw [thick,->] (0,-1.05) -- (0,2.05);
		\draw [thick,->] (-0.05,0) -- (5.05,0);
		\draw (5,0) node[right]{$\theta$};
		\draw (0,2) node[above]{$\Phi(\theta)$};
		\filldraw[color=black] (2.061553,-0.4519411) circle(0.4ex);
		\draw [densely dotted,-] (2.061553,-0.4519411) -- (2.061553,0) node[above] {\small $\vartheta$};
		\draw (0,0) -- (4.8,-0.3428571);
		\filldraw[color=black] (3.5,-0.25) circle(0.4ex);
		\draw [densely dotted,-] (3.5,-0.25) -- (3.5,0) node[above] {\small $\gamma$};		
		\draw [thick, domain=0:5, samples=128, color=black] plot(\x, {(\x/2-1.25)^2-0.5});
\end{tikzpicture}
\end{center}
\subcaption{\small The case covered by Theorem \ref{Thm:beyond the boundary}.}
\end{subfigure}
\caption{\small The three regimes that can occur.}
\end{figure}
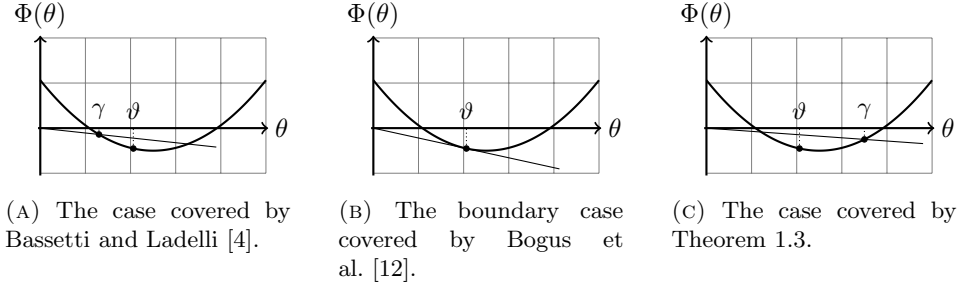

\begin{definition}
For $\gamma \in(0,2]$ by $\meas^1_\gamma(\R)$ we denote the class of probability measures
with finite absolute moment of order $\gamma$, centered if $\gamma>1$.
\end{definition}
The key property of the class $\meas^1_\gamma(\R)$ is
that for a sequence $(X_j)_{j \in \N}$ of independent random variables with
$\Law (X_j) \in \meas^1_\gamma(\R)$, we have
\begin{align}	\label{eq:moment subadditivity}
\E\bigg[\Big|\sum_{j =1}^\infty X_j\Big|^\gamma\bigg] \leq 2 \sum_{j=1}^\infty \E\big[ |X_j|^\gamma\big], \qquad n\in\N
\end{align}
if the right-hand side is finite.
In fact, for $\gamma \leq 1$, \eqref{eq:moment subadditivity} holds,
even with $2$ replaced by $1$, by the subadditivity of $x \mapsto |x|^\gamma$.
On the other hand, for $\gamma \in (1,2]$, \eqref{eq:moment subadditivity} is
a consequence of the von Bahr--Esseen inequality.
We conclude that if $\Phi(\gamma)<\infty$,
then the restriction $Q\!\restriction_{\meas^1_\gamma(\R)}$ is a well-defined mapping
from $\meas^1_\gamma(\R)$ to itself.

We now state the two main results of the paper.

\begin{theorem}	\label{Thm:existence and uniqueness}
Each of the following assumptions is sufficient for the existence of a solution
$(\mu_t)_{t \geq 0}$ to the evolution equation \eqref{eq:kinetic-type equation}:
\begin{enumerate}[(i)]
	\item
		Assumption (A\ref{enum:finitness of BRW}) holds.
	\item
		There exists a $\gamma\in(0,2]$ with $\mu_0\in\meas^1_\gamma(\R)$ and $\Phi(\gamma)<\infty$.
\end{enumerate}
If (i) holds, the solution is unique, if (ii) holds it is the unique solution
satisfying $\sup_{0 \leq s \leq t} \int |x|^\gamma \, \mu_s(\dx)<\infty$ for all $t \geq 0$.
\end{theorem}

\begin{theorem}	\label{Thm:beyond the boundary}
Suppose that (A\ref{enum:non-lattice}) through (A\ref{enum:XlogX condition}) hold with $0 < \vartheta < 2$.
Further, assume that the initial condition $\mu_0$ belongs to the class $\meas^1_\gamma(\R)$
for some $\gamma \in (\vartheta,2]$.
Then there is a solution $(\mu_t)_{t \geq 0}$ to \eqref{eq:kinetic-type equation}
and a probability measure $\mu_\infty$ on the Borel sets of $\R$,
not concentrated in a single point, such that
\begin{equation*}
\lim_{t \to \infty} \phi_t\big(t^{\frac3{2\vartheta}} e^{-F(\vartheta) t} \xi\big) \to \phi_\infty(\xi)
\quad	\text{ for all } \xi \in \R,
\end{equation*}
for the characteristic function $\phi_\infty$ of $\mu_\infty$.

Moreover, if $Z$ is a random variable with law $\mu_{\infty}$, then it satisfies
the following stochastic fixed-point equation
\begin{equation}	\label{eq:smoothing equation}
Z \eqdist U^{F(\vartheta)} \sum_{j=1}^{N}
A_j Z^{(j)},
\end{equation}
where $Z^{(1)}, Z^{(2)}, \ldots$ are independent copies of $Z$, $U$ is uniformly distributed on $(0,1)$
and $U$, $A = (A_1,A_2,\ldots)$
and $(Z^{(j)})_{j \in \N}$ are independent.
\end{theorem}

Equation \eqref{eq:smoothing equation} is called fixed-point equation of the smoothing transformation.
A lot of information on $Z$ can be extracted from the fact that (the law of) $Z$
satisfies \eqref{eq:smoothing equation}.
More precisely, we are in the situation with nonnegative weights $U^{F(\vartheta)} A_j$, $j \in \N$
and possibly real-valued $Z$. In this setup, the equation has been solved in \cite{Alsmeyer+Meiners:2013}.
An important parameter for \eqref{eq:smoothing equation} is the characteristic index $\alpha > 0$,
the minimal positive solution of the equation $m(t) = 1$ where
\begin{equation*}
m(t) = \E\bigg[\sum_{j \geq 1} U^{t F(\vartheta)} A_j^t \bigg] = \frac{\Phi(t)+1}{\frac{t}{\vartheta} \Phi(\vartheta) + 1}.
\end{equation*}
Notice that $m(\vartheta) = 1$. If $\Phi(t) = \infty$ for all $t < \vartheta$, then $\alpha=\vartheta$.
If $\Phi(t) < \infty$ for some $0 \leq t < \vartheta$,
then $m$ is differentiable (from the left) at $t=\vartheta$.
The derivative equals $0$ at $\vartheta$ if \eqref{eq:boundary case assumption} holds.
By the convexity of $m$, we again infer that $\alpha=\vartheta$.
If $m(t) < \infty$ for some $t<\vartheta$, Theorems 2.1 and 2.2 of \cite{Alsmeyer+Meiners:2013} imply that
\begin{equation*}
Z \eqdist W^{1/\vartheta} Y_\vartheta
\end{equation*}
where $W \geq 0$ is the limit of an associated derivative martingale
(namely, that of the branching random walk with first generation positions given by $-\log(U^{\Phi(\vartheta)} A_j^\vartheta)$, $j \in \N$
with $A_j > 0$)
and $Y_\vartheta$ is
a strictly $\vartheta$-stable random variable independent of $W$.
From this representation,
one can deduce various properties of the distribution of $Z$.
For instance, one may deduce the tail-behavior of $Z$ from that of $W$ and $Y_\vartheta$
using the main result of the recent paper \cite{Buraczewski+al:2020}.

Let us demonstrate how Theorem \ref{Thm:beyond the boundary} translates to the particular economical model
described above  by the random vector $(A_1,A_2)$ defined in \eqref{eq:econophysics model}.
In this model, $\Phi(1)=0$ and if $\Phi(r)=0$ for some $r<1$,
then the typical agent goes bankrupt (provided the initial wealth distribution $\mu_0$ has finite mean).
In this case, the minimizer $\vartheta$ is in the interval $(r,1)$ and $F(\vartheta)<0$.
If the tail of the initial distribution is heavy enough (i.e., if $\gamma<\vartheta$),
then the wealth of a typical agent behaves like $e^{F(\gamma)t}$,
which diverges for $\gamma<r$.
Next, in the boundary case $\gamma=\vartheta$, the correct asymptotic is $t^{-\frac1{2\vartheta}} e^{F(\vartheta) t}$.
Our result deals with the remaining case where the tail of the initial distribution is not heavy enough, i.e.,
$\gamma > \vartheta$.
In particular, this covers the situation where the first moment exists,
which seems to be the most natural case in this context.
In this case, the wealth of a typical agent decays like $t^{-\frac3{2\vartheta}} e^{F(\vartheta) t}$ as $t \to \infty$.

\subsection{Representation of solutions: the branching random walk connection.}

Given an initial distribution $\mu_0$ and the random vector $A$,
we give a representation of $\mu_t$ as the law of a continuous-time branching random walk at time $t$.
The exact form of representation was developed in \cite{Bogus+al:2019},
see also \cite{Bassetti+Ladelli:2012} and the references therein for earlier results.

We write $\I \defeq \bigcup_{n \in \N_0} \N^n$
where $\N^0 = \{\varnothing\}$ contains only the empty tuple $\varnothing$.
For $u \in \I$, $u=(u_1,\ldots,u_m)$, we also write $u_1 \ldots u_m$
and if $v=(v_1,\ldots,v_n)$, we write $uv$ for $(u_1,\ldots,u_m,v_1,\ldots,v_n)$.
Further, if $k \leq m$, we set $u|_k \defeq u_1 \ldots u_k$.
Finally, for $u \in \I$, we use the notation $|u|=n$ for $u \in\N^n$.

Throughout the paper, we work on a fixed probability space $(\Omega,\F,\Prob)$
on which two independent families $(A(u),E(u))_{u \in \I}$ and $(X_u)_{u \in \I}$
of random vectors and random variables, respectively, are defined such that
\begin{itemize}
	\item
		the $(A(u),E(u))$, $u \in \I$ are independent and identically distributed (i.i.d.)
		copies of $(A,E)$ where $A$ is given and $E$ is an independent unit-mean exponential random variable;
	\item the $X_u$, $u \in \I$ are i.i.d.\ copies of a random variable $X$ with $\Law(X) = \mu_0$.
\end{itemize}
For convenience, we denote quantities related to the ancestor without the label $\varnothing$,
i.e., $(A,E) = (A(\varnothing),E(\varnothing))$ etc.

We now recursively define a continuous-time Markov branching process $(\mathcal{Y}_t)_{t\geq 0}$
starting with one particle,
the ancestor, denoted by $\varnothing$,
at time $t=0$. The birth-time of the ancestor is $\sigma(\varnothing) = 0$.
If a particle labelled $u \in \I$ is born at time $\sigma(u)$,
it lives an exponential lifetime $E(u)$ until $\sigma(u)+E(u)$
at which time it dies and simultaneously gives birth to new particles labelled $u1,u2,\ldots$.
For a particle $u = u_1 \ldots u_m \in \I$, we write
\begin{equation*}
S(u) \defeq - \sum_{k=1}^{m}  \log A_{u_k} (u|_{k-1})
\end{equation*}
for its position on the real line. The position $S(u) = \infty$, we consider as a ghost type:
the corresponding individual is never born.
We write
\begin{equation*}
\I_t \defeq \{u \in \I: S(u) < \infty\text{ and } \sigma(u) \leq t < \sigma(u)+E(u)\}
\end{equation*}
for the set of labels pertaining to individuals alive at time $t$.
Finally, we write
\begin{equation*}
\cZ_t \defeq \sum_{u \in \I_t} \delta_{S(u)}
\end{equation*}
for the continuous-time branching random walk at time $t \geq 0$.
Throughout the paper, we denote by $(T_n)_{n \in \N_0}$
the sequence of points in increasing order of a homogeneous Poisson process
with intensity $1$ and a point at the origin, i.e., $T_0=0$.
We suppose that $(T_n)_{n \in \N_0}$ is independent of the $(A(u),E(u))$, $X_u$, $u \in \I$.
The Laplace transform at $\theta \geq 0$ of the intensity measure of $\cZ_{t}$ is given by
\begin{align}	\label{eq:m(t,theta)}
m(t,\theta)
&\defeq \E \bigg[\sum_{u \in \I_t} e^{-\theta S(u)} \bigg] = \E \bigg[\sum_{n\ge0} \sum_{|u|=n} e^{-\theta S(u)}\1_{\{\sigma(u)\le t,\sigma(u)+E(u)> t\}} \bigg] \notag	\\
&= \sum_{n \geq 0} \E \bigg[ \sum_{|u|=n} e^{-\theta S(u)}\bigg] \Prob(T_n \leq t < T_{n+1})	\notag	\\
&= \sum_{n \geq 0} (\Phi(\theta)+1)^n e^{-t}\frac {t^n}{n!} = e^{t\Phi(\theta)}.
\end{align}
By classical results \cite[Theorem 2.1, p.~119]{Asmussen+Hering:1983}
it follows that the set $\I_t$ is finite almost surely for all $t \geq 0$,
provided (A\ref{enum:finitness of BRW}) holds. In particular, the sum
\begin{equation}	\label{eq:random sum}
U_t  \defeq \sum_{u \in \I_t} e^{-S(u)} X_u
\end{equation}
is a well-defined, finite random variable.
On the other hand, if $\Law(X)\in\meas^1_\theta(\R)$ and $\Phi(\theta)<\infty$ for some $0< \theta \leq 2$,
then the right hand side of \eqref{eq:random sum} converges in $L^\theta$ by \eqref{eq:moment subadditivity} and \eqref{eq:m(t,theta)}.

The connection between the continuous-time branching random walk $\cZ_t$ and the kinetic-type
evolution equation \eqref{eq:kinetic-type equation} is established in the following proposition,
which implies Theorem \ref{Thm:existence and uniqueness}.

\begin{proposition}	\label{Prop:Wild representation}
In the situation of Theorem \ref{Thm:existence and uniqueness},
each of the conditions (i) and (ii) of the theorem implies
the existence of a solution $(\mu_t)_{t \geq 0}$ to \eqref{eq:kinetic-type equation}
given by
\begin{equation}	\label{eq:Wild representation}
\mu_t = \Law\big(U_t \big),	\quad	t \geq 0.
\end{equation}
If (i) holds, the solution is unique.
If (ii) holds, then it is the unique solution
satisfying $\sup_{0 \leq s \leq t} \int |x|^\gamma \, \mu_s(\dx)<\infty$ for all $t \geq 0$.
\end{proposition}
\begin{proof}
First, we provide an equation, which is equivalent to \eqref{eq:kinetic-type equation FT}
and easier to work with.
If $\phi_t$ is a solution to the kinetic-type equation \eqref{eq:kinetic-type equation FT},
then it satisfies the integral equation
\begin{align*}
e^t\phi_t(\xi)-\phi_0(\xi) = \int_0^te^s\widehat Q(\phi_s)(\xi) \, \ds,	\quad	t \geq 0,\;\xi \in \R,
\end{align*}
and vice versa.
Recall that $E = E(\varnothing)$ is a unit-mean exponential random variable.
With the convention that $\phi_t=\phi_0$ for $t \leq 0$, in view of \eqref{eq:Q_hat},
the above equation can be rewritten as
\begin{align}	\label{eq:kinetic-type equation FT version 2}
\phi_t(\xi) = e^{-t}\phi_0(\xi) + \int_0^t e^{-s}\widehat Q(\phi_{t-s})(\xi) \, \ds
= \E\bigg[\prod_{u\in \I_t^1} \phi_{t-E}(e^{-S(u)} \xi)\bigg],
\end{align}
valid for $t \geq 0$ and $\xi \in \R$
where
\begin{equation*}
\I_t^n = \I_t^{n,1} \cup \I_t^{n,2}\defeq\{u\in \I_t:  |u| \leq n\} \cup \{ u\in \I: \sigma(u) \leq t, |u|=n\}
\end{equation*}
for any $n \in \N_0$.
We show that the  function $\psi_t(\xi)\defeq\E[\exp(\imag \xi U_t)]$
satisfies \eqref{eq:kinetic-type equation FT version 2}
provided that condition (i) or (ii) of Theorem \ref{Thm:existence and uniqueness} holds.
Indeed, for any $t \geq 0$,
\begin{align*}
\E\big[\exp(\imag\xi U_t)\big|(A,E)\big]
= \1_{\{E > t\}} \E\big[\exp(\imag\xi X)\big]
+ \1_{\{E \leq t\}}\prod_{|u|=1}\psi_{t-E}( e^{-S(u)} \xi),
\end{align*}
and therefore\footnote{
Notice that the a.\,s.\ finiteness of $U_t$ and $N<\infty$ a.\,s.\ are the only assumptions
required to draw this conclusion.}
\begin{align*}
\psi_t(\xi)
&= \E\bigg[\1_{\{E > t\}} \phi_0(\xi) + \1_{\{E \leq t\}} \prod_{|u|=1}\psi_{t-E}(e^{-S(u)}\xi)\bigg]
= \E\bigg[\prod_{u\in \I_t^1} \psi_{t-E}(e^{-S(u)} \xi)\bigg].
\end{align*}
Let us also note that if $\mu_0\in\meas^1_\gamma(\R)$ and $\Phi(\gamma)<\infty$, then
by \eqref{eq:moment subadditivity} and \eqref{eq:m(t,theta)} we infer
\begin{equation*}	\textstyle
\E[|U_t|^\gamma] \leq 2 e^{t\Phi(\theta)} \int |x|^\gamma \, \mu_0(\dx),
\end{equation*}
which is locally bounded. Moreover,  $\E[U_t]=0$ for all $t \geq 0$ if $\gamma>1$.

\noindent
Now we prove that $\psi_t$ is the only solution to \eqref{eq:kinetic-type equation FT version 2}.
Let $(\phi_t)_{t \geq 0}$ be any solution with initial condition $\phi_0$.
Inductively, for any $n \in \N_0$,
iterating \eqref{eq:kinetic-type equation FT version 2} we get
\begin{align*}
\phi_t(\xi) &= \E\bigg[\prod_{u\in \I_t^n} \phi_{t-\sigma(u)}(e^{-S(u)} \xi)\bigg]	\\
&= \E\bigg[\prod_{u\in \I_t^{n,1}} \phi_{t-\sigma(u)}(e^{-S(u)} \xi) \cdot \prod_{u\in \I_t^n \setminus \I_t^{n,1}} \phi_{t-\sigma(u)}(e^{-S(u)} \xi)\bigg].
\end{align*}
We show that, for fixed $t \geq 0$,
\begin{align}
\label{eq:product tends to 1}
\prod_{u\in \I_t^n \setminus \I_t^{n,1}} \phi_{t-\sigma(u)}(e^{-S(u)}\xi)
\to 1	\quad \text{in }L^1 \text{ as } n \to \infty.
\end{align}
This is clear if (i) holds since then, with probability one, $\I_t$ is finite
and hence the product above is eventually indexed by the empty set.
On the other hand, if (ii)
and the additional assumption $\sup_{0 \leq s \leq t} \int |x|^\gamma \, \mu_s(\dx)<\infty$ for all $t \geq 0$ hold,
using \cite[Theorem 1 on p.\;295]{Chow+Teicher:1997}
we infer existence of a function $t\mapsto C(t) \geq 0$ such that
\begin{equation*}
\sup_{0 \leq s \leq t} |1-\phi_s(\xi)| \leq C(t) |\xi|^\gamma	\quad	\text{for all } \xi \in \R.
\end{equation*}
Using this together with $\I_t^n \setminus \I_t^{n,1} = \{u \in \I: |u|=n, S(u) + E(u) \leq t\}$
and the elementary inequality $|1-\prod_k z_k| \leq \sum_k |1-z_k|$,
valid for $z_k \in \C$ with $|z_k| \leq 1$,	
we conclude
\begin{align*}
\bigg|1-\prod_{{u\in \I_t^n \setminus \I_t^{n,1}}} \phi_{t-\sigma(u)}(e^{-S(u)}\xi) \bigg|
\leq C(t) |\xi|^\gamma \sum_{\substack{|u|=n, \\ \sigma(u) + E(u) \leq t}} e^{-\gamma S(u)}.
\end{align*}
Since
\begin{equation*}
\E\bigg[\sum_{\substack{|u|=n, \\ \sigma(u) + E(u) \leq t}} e^{-\gamma S(u)}\bigg]
\leq (\Phi(\gamma)+1)^n \sum_{k > n} e^{-t}\frac{t^{k}}{k!} \to 0,
\end{equation*}
as $n$ goes to infinity, we conclude \eqref{eq:product tends to 1}.
Consequently, in both cases we have
\begin{align*}
\phi_t(\xi)
&= \lim_{n\to\infty} \E\bigg[\prod_{u\in \I_t^n} \phi_{t-\sigma(u)}(e^{-S(u)} \xi)\bigg]
= \lim_{n\to\infty} \E\bigg[\prod_{u\in \I_t^{n,1}} \phi_{0}(e^{-S(u)} \xi)\bigg]	\\
& =\lim_{n\to\infty} \E\bigg[\exp\bigg(\imag\xi\sum_{u\in \I_t^{n,1}} e^{-S(u)} X_u\bigg)\bigg]
=\E\big[\exp(\imag \xi U_t)\big]
=\psi_t(\xi).
\end{align*}
\end{proof}

The above result provides an explicit form of the solution to Equation \eqref{eq:kinetic-type equation}.
Therefore, in order to prove our main result we need to find an appropriate
scaling of the random sum \eqref{eq:random sum}
leading to a nontrivial limit law as $t \to \infty$. For this purpose,
first applying the Croft-Kingman lemma \cite{Kingman:1963},
we reduce the problem of describing convergence along any sequence to convergence along arbitrary lattice sequences
(Section \ref{sec:reduction to the lattice case}).
Finally, we show the existence of the limit along lattice sequences (Section \ref{sec:lattice convergence}).

\section{Reduction to the lattice case}	\label{sec:reduction to the lattice case}

The goal of this section is to prove the following lemma.

\begin{lemma}	\label{Lem:lattice to arbitrary convergence}
Suppose that   (A\ref{enum:non-lattice}) through (A\ref{enum:XlogX condition}) holds,
$\mu_0\in\meas^1_\gamma(\R)$ for some $\gamma \in (\vartheta,2]$
and that, for any fixed $\delta >0$,
\begin{equation}	\label{eq:random sum scaled lattice}
(n\delta)^{\frac3{2\vartheta}} e^{-F(\vartheta) n\delta} U_{n\!\delta}	\distto	Z_\delta
\quad	\text{as } n \to \infty
\end{equation}
for some non-degenerate random variable $Z_\delta$. Then
\begin{equation}	\label{eq:random sum scaled limit}
t^{\frac3{2\vartheta}} e^{-F(\vartheta) t} U_t \distto Z_1
\quad	\text{as } t \to \infty.
\end{equation}
Moreover, the   random variable $Z\defeq Z_1$ satisfies \eqref{eq:smoothing equation}.
\end{lemma}

The lemma above is proved in several steps.
First, for $p > 0$ and $x \geq 0$, we define $f_p(x) \defeq (1+\log_+^px)x$
and notice that $f_p$ is nearly submultiplicative
in the sense that
\begin{equation}	\label{eq:f_p submultiplicative}
f_p(xy) \leq 2^p f_p(x)f_p(y)		\quad \text{for all }	x,y \geq 0.
\end{equation}
Similarly, $f_p$ is subadditive up to a multiplicative constant, namely,
\begin{equation}	\label{eq:f_p subadditive}
f_p(x+y) \leq 2(1+\log^p 2) (f_p(x)+f_p(y))		\quad \text{for all }	x,y \geq 0.
\end{equation}
Further, let
\begin{align}	\label{eq:h_p}
h_p(x)	&\defeq	\int_0^x \Big(\frac{p}{e}\Big)^pt\1_{[0,e^p]}(t) + (\log t)^p\1_{(e^p,\infty)}(t) \, \dt.
\end{align}
Then $h_p$ is convex with concave derivative $h'_p$.
Further, $f_p$ and $h_p$ are asymptotically equivalent, i.e.,
\begin{equation}	\label{eq:f_p sim h_p}
\lim_{x \to \infty} \frac{f_p(x)}{h_p(x)} = 1.
\end{equation}
Consequently, since $h_p'(0)=0<1=f_p'(0)$, there is some $C_p>0$ such that
\begin{equation}	\label{eq:h_p leq C_p f_p}
h_p(x) \leq C_p f_p(x)	\quad	\text{for all } x \geq 0.
\end{equation}
We start with a technical lemma.

\begin{lemma}	\label{lem:XlogX}
Suppose that
\begin{align*}
&\E\bigg[f_p\bigg(\sum_{|u|=1} e^{-S(u)}\bigg)\bigg] < \infty\text{ and }
\E\bigg[\sum_{|u|=1}f_p\big( e^{-S(u)}\big)\bigg] < \infty
\end{align*}
for some $p > 0$. Then there is a constant $C> 0$ such that, for any $t \geq 0$,
\begin{align}
\E \bigg[\sum_{u\in\I} f_p\big(e^{-S(u)}\big)\1_{\{\sigma(u) \leq t\}}\bigg]	&\leq C e^{Ct}	\label{eq:sum xlogx}	\\
\text{and}	\qquad
\E\bigg[f_p\bigg(\sum_{u \in \I_t} e^{-S(u)}\bigg)\bigg]				&\leq C e^{Ct}.	\label{eq:xlogx sum}
\end{align}
\end{lemma}
\begin{proof}
Throughout the proof, if some quantity depending on $t \geq 0$ is bounded by $C e^{Ct}$ for all $t \geq 0$
and some constant $C>0$, then we say that the quantity grows at most exponentially fast.
Using \eqref{eq:f_p submultiplicative} and induction on $n$, we infer
\begin{align}
\E\bigg[\sum_{|u|=n}f_p\big(e^{-S(u)}\big)\bigg] \leq q^n,
\end{align}
for some $q>1$. Then
\begin{align}
\E \bigg[\sum_{u\in\I}f_p\big(e^{-S(u)}\big)\1_{\{\sigma(u) \leq t\}}\bigg]
&= \sum_{n \geq 0}\E \bigg[\sum_{|u|=n}f_p\big(e^{-S(u)}\big) \1_{\{\sigma(u) \leq t\}} \bigg]	\label{eq:2.11}	\\
&\leq \sum_{n \geq 0}q^n \Prob(T_n \leq t) = \frac{qe^{t(q-1)}-1}{q-1}<\infty,	\notag
\end{align}
proving \eqref{eq:sum xlogx}.
Turning to the proof of \eqref{eq:xlogx sum}, we first notice that, for every $t \in \R$,
\begin{equation*}
\sum_{u \in \I_t} e^{-S(u)} \leq \sum_{u \in \I} e^{-S(u)} \1_{\{\sigma(u) \leq t\}} \eqdef M^t.
\end{equation*}
(Here, for $t < 0$, both sums are empty and hence have value $0$.)
Since $f_p$ is monotone, it suffices to prove that $\E[f_p(M^t)]$ grows at most exponentially fast.
Since $f_p(x) \geq x$ for any $x \geq 0$, we conclude from \eqref{eq:sum xlogx}
that $H(t) \defeq \E[M^t] \leq C'e^{C't}$ for all $t \geq 0$
and an appropriate constant $C'>0$.
Thus, by \eqref{eq:f_p subadditive}, for all $t \geq 0$,
\begin{align*}
\E[f_p(M^t)] &= \E[f_p(M^t-H(t) + H(t))]	\\
&\leq 2(1+\log^p 2) \big(\E[f_p(|M^t-H(t)|)] + f_p(H(t))\big)	\\
&\leq 2(1+\log^p 2) \big(\E[f_p(|M^t-H(t)|)] + f_p(C'e^{C't})\big).
\end{align*}
Therefore, it suffices to prove that $\E[f_p(|M^t-H(t)|)]$ grows at most exponentially fast in $t$.
By \eqref{eq:f_p sim h_p}, there is a constant $C''$ such that
\begin{equation*}
\E[f_p(|M^t-H(t)|)] \leq 2\E[h_p(|M^t-H(t)|)] + C''	\quad	\text{for all }	t \geq 0,
\end{equation*}
so it suffices to prove that $\E[h_p(|M^t\!-\!H(t)|)]$ grows at most exponentially fast in $t$.
To this end, let $\varnothing = u_1,u_2, \ldots \in \I$ be a deterministic enumeration of $\I$
such that, with $\I^{(n)}=\{u_1,\ldots,u_n\}$, the sequence $(\I^{(n)})_{n \in \N}$ is a strictly increasing
sequence of subtrees of $\I$, and let
\begin{equation*}
\F_{\I^{(n)}} \defeq \sigma((A(u),E(u)): u \in \I^{(n)}).
\end{equation*}
Then $M^t_n \defeq \E[M^t | \F_{\I^{(n)}}]$, $n \in \N$ is a (uniformly integrable) martingale.
By the martingale convergence theorem, $M^t_n \to M^t$ a.\,s.\ and in $L^1$.
Consider the martingale differences
\begin{align*}
M^t_n-M^t_{n-1}
&= \sum_{u \in \I} \big(\E[e^{-S(u)}\1_{\{\sigma(u) \leq t\}} | \F_{I^{(n)}}]
- \E[e^{-S(u)}\1_{\{\sigma(u) \leq t\}} | \F_{I^{(n-1)}}]\big),
\ n \in \N.
\end{align*}
If $u$ is not a strict descendant of $u_n$, i.e., if there is no $v \in \I$ with $|v| \geq 1$ such that
$u = u_n v$, then $(A(u_n),E(u_n))$ is independent
of the $\sigma$-algebra generated by $e^{-S(u)}\1_{\{\sigma(u) \leq t\}}$ and $\F_{I^{(n-1)}}$,
hence, $\E[e^{-S(u)} | \F_{I^{(n)}}] = \E[e^{-S(u)} | \F_{I^{(n-1)}}]$ a.\,s.
Hence, with
\begin{align*}
D_n^t
&\defeq \E\bigg[\sum_{|v| \geq 1} e^{-S(u_n v)}\1_{\{\sigma(u_n v) \leq t\}} | \F_{I^{(n)}} \bigg]	\\
&= \sum_{j=1}^{N(u_n)} e^{-S(u_nj)} \E\bigg[\sum_{v \in \I} e^{-(S(u_njv)-S(u_nj))} \1_{\{\sigma(u_njv)-\sigma(u_nj) \leq t - \sigma(u_nj)\}} \Big| \F_{\I^{(n-1)}}\bigg]	\\
&= \sum_{j=1}^{N(u_n)} e^{-S(u_nj)} H(t-\sigma(u_nj))	\quad	\text{a.\,s.,}
\end{align*}
we have
\begin{equation}	\label{eq:M^t_n-M^t_n-1=D_n^t-E[D_n^t|F_n-1]}
M^t_n-M^t_{n-1} = D_n^t - \E[D_n^t | \F_{\I^{(n-1)}}]	\quad	\text{a.\,s.}
\end{equation}
Since $h_p(0)=0$, $h_p$ is increasing and convex with concave derivative,
we may apply the Topchi\u\i-Vatutin inequality \cite{Alsmeyer+Roesler:2003}
and infer
\begin{align*}
\E[h_p(|M^t-H(t)|)] &\leq 2 \sum_{n=1}^\infty \E[h_p(|M^t_n-M^t_{n-1}|)]	\\
&\leq 2 \sum_{n=1}^\infty \big(\E[h_p(D_n^t)] + \E[h_p(\E[D_n^t | \F_{\I^{(n-1)}}])]\big)	\\
&\leq 4 \sum_{n=1}^\infty \E[h_p(D_n^t)]
\leq 4 C_p \sum_{n=1}^\infty \E[f_p(D_n^t)],
\end{align*}
where we have applied Jensen's inequality for conditional expectations in the next-to-last step
and \eqref{eq:h_p leq C_p f_p} in the last step. Here, using the definition of $D_n^t$,
$H(t) \leq C' e^{C't} \1_{[0,\infty)}(t)$ for all $t \in \R$, and \eqref{eq:f_p submultiplicative} (twice), we find
\begin{align*}
\sum_{n=1}^\infty \E[f_p(D_n^t)]
&= \sum_{u \in \I} \E\bigg[f_p\bigg(\sum_{j=1}^{N(u)} e^{-S(uj)} H(t-\sigma(uj))\bigg)\bigg]	\\
&\leq \sum_{u \in \I} \E\bigg[f_p\bigg(\sum_{j=1}^{N(u)} e^{-S(uj)} \1_{\{\sigma(uj) \leq t\}} C' e^{C't}\bigg)\bigg]	\\
&\leq 2^p f_p(C' e^{C't}) \sum_{u \in \I} \E\bigg[f_p\bigg(\sum_{j=1}^{N(u)} e^{-S(uj)} \1_{\{\sigma(uj) \leq t\}} \bigg)\bigg]	\\
&\leq 4^p f_p(C' e^{C't}) \sum_{u \in \I} \E\bigg[f_p(e^{-S(u)}) \1_{\{\sigma(u) \leq t\}} f_p\bigg(\sum_{j=1}^{N(u)} e^{-(S(uj)-S(u))} \bigg)\bigg]	\\
&= 4^p f_p(C' e^{C't}) \E \bigg[\sum_{u \in \I} f_p(e^{-S(u)}) \1_{\{\sigma(u) \leq t\}} \bigg]
\E\bigg[ f_p\bigg(\sum_{|u|=1} e^{-S(u)} \bigg)\bigg],
\end{align*}
which is finite and grows at most exponentially fast by \eqref{eq:2.11} and since the last expectation
is finite by assumption.
\end{proof}

\begin{lemma}	\label{Lem:L^r-continuous}
Suppose that the assumptions of Lemma \ref{Lem:lattice to arbitrary convergence} hold.
 Then the family $(U_t)_{t \geq 0}$ is continuous in $L^\vartheta$,
i.e., $\E[|U_t-U_s|^\vartheta] \to 0$ as $s \to t$.
In particular, if $a:[0,\infty) \to [0,\infty)$ is a deterministic nonnegative continuous function,
then also $(a(t) U_t)_{t \geq 0}$ is continuous in $L^\vartheta$.
\end{lemma}
\begin{proof}
Define $g(t,s) \defeq \E[|U_t-U_s|^\vartheta]$ for $s,t \geq 0$.
We first show that $g(t,0) \to 0$ as $t \to 0$.
Notice that $U_0  = X_\varnothing = X \sim \mu_0$. Then,
with $S_t \defeq \{E \leq t\}$ denoting the event that there was a split
in the interval $[0,t]$, we have $\Prob(S_t) = 1-e^{-t}$.
Consequently,
\begin{align*}
g(t,0) &= \E[|U_t-X|^\vartheta \1_{S_t}]
\leq 2^\vartheta \E[|U_t|^\vartheta \1_{S_t}] + 2^\vartheta\E[|X|^\vartheta] (1-e^{-t}).
\end{align*}
The second summand on the right-hand side vanishes as $t \to 0$, so it remains to consider the first one.
For the function $h_1$ defined by \eqref{eq:h_p} with $p=1$
the expectation $\E[h_1(|U_t|^\vartheta)]$ remains bounded as $t$ goes to 0.
We postpone the proof of this fact and first show how it implies $\E[|U_t|^\vartheta \1_{S_t}]  \to 0$ as $t \to 0$.
Indeed, since $h_1$ is convex and grows superlinearly fast,
the Legendre-Fenchel transform $h_1^*(y) \defeq \sup_{x \geq 0} (xy - h_1(x))$
of the function $h_1$ is finite on the half-line $[0,\infty)$
and $h^*_1(y)\to\infty$ as $y \to \infty$.
From the definition of $h_1^*$, we conclude that $xy\le h_1(x)+h^*_1(y)$
for all $x,y \geq 0$ (a generalized version of Young's inequality) and $h_1^*(y) > 0$ iff $y > 0$.
Using these inequalities, \eqref{eq:h_p leq C_p f_p} and \eqref{eq:f_p submultiplicative},
we infer, for any $s_t > 1$,
\begin{align*}
\E[|U_t|^\vartheta \1_{S_t}]& \le \E[h_1(s_t^{-1}|U_t|^\vartheta)]+\E[h^*_1(s_t \1_{S_t})] \\
&\leq s_t^{-1}C\E[f_1(|U_t|^\vartheta)]+h^*_1(s_t)\Prob(S_t)
\end{align*}
where $C>0$ is an appropriate constant.
Taking now $s_t \to \infty$ such that $h_1^*(s_t)=\Prob(S_t)^{-1/2} = (1-e^{-t})^{-1/2} \to 0$ as $t \to 0$,
we conclude that the second summand tends to $0$ as $t \to \infty$.
Regarding the first, notice that \eqref{eq:f_p sim h_p} together with $\limsup_{t \to 0} \E[h_1(|U_t|^\vartheta)] < \infty$
implies that it also tends to 0 as $t \to 0$.

\noindent
We now turn to the proof of the fact that $\limsup_{t \to 0} \E[h_1(|U_t|^\vartheta)] < \infty$.
For technical reasons, we need to replace $h_1(|x|^\vartheta)$ by a function of the same order of growth
with more convenient properties.
To this end, first suppose that $\vartheta \in (1,2)$ and consider $g_\vartheta'': [0,\infty) \to [0,\infty)$
\begin{align*}
g_\vartheta''(u) \defeq	\begin{cases}
					\frac{1}{e(2-\vartheta)},	&	\text{ for } u \leq e^{1/(2-\vartheta)},	\\	
					u^{\vartheta - 2} \log u,	&	\text{ for } u \geq e^{1/(2-\vartheta)}.
					\end{cases}
\end{align*}
The function $g_\vartheta''$ is nonnegative, continuous and non-increasing,
hence $g_\vartheta: [0,\infty) \to [0,\infty)$, defined by
\begin{equation*}
g_\vartheta(x) = \int_0^x \int_0^t g_\vartheta''(u) \, \du \, \dt,	\quad	x \geq 0,
\end{equation*}
is convex with concave derivative.
Two applications of the direct half of Karamata's theorem \cite[Proposition 1.5.8]{Bingham+Goldie+Teugels:1987}
imply that
\begin{equation*}
g_\vartheta(x) \sim \frac{x^{\vartheta} \log x}{\vartheta(\vartheta - 1)} 	\quad	\text{as } x \to \infty.
\end{equation*}
Since $h_1(x^\vartheta) \sim \vartheta x^\vartheta \log x$ as $x \to \infty$, we have
$\limsup_{t \to 0} \E[h_1(|U_t|^\vartheta)] < \infty$ iff $\limsup_{t \to 0} \E[g_\vartheta(|U_t|)] < \infty$
Further, since we also have $f_1(x^\vartheta) \sim x^\vartheta$ as $x \to 0$
and $f_1(x^\vartheta) \sim \vartheta x^\vartheta \log x$ as $x \to \infty$,
whereas $g_\vartheta(x) \sim x^{2}/(2e(2-\vartheta))$ as $x \to 0$,
there exists a constant $C_{\vartheta} > 0$
such that $g_\vartheta(x) \leq C_\vartheta f_1(x^\vartheta)$ for all $x \geq 0$.
A combination of this inequality, the (conditional) Topchi\u\i-Vatutin inequality (recall that $\E[X]=0$ in this case)
and \eqref{eq:f_p submultiplicative} yields
\begin{align*}
\E[g_\vartheta(|U_t|)] &\leq 2 \E \bigg[\sum_{u \in \I_t} g_\vartheta(e^{- S(u)} |X_u|)\bigg]
\leq 2 C_\vartheta \E \bigg[\sum_{u \in \I_t} f_1(e^{-\vartheta S(u)} |X_u|^\vartheta)\bigg]	\\
&\leq 4 C_\vartheta \E \bigg[\sum_{u \in \I_t} f_1(e^{-\vartheta S(u)}) f_1(|X_u|^\vartheta)\bigg],
\end{align*}
which is bounded by Lemma \ref{lem:XlogX} and $\E[|X|^\gamma]<\infty$.
If $\vartheta = 1$, the situation is easier and the above argument works with
$g_\vartheta(x) \defeq h_1(x)$, $x \geq 0$ as this function is convex with concave derivative.
If $\vartheta<1$, then we define $g_\vartheta':[0,\infty) \to [0,\infty)$ via
\begin{align*}
g_\vartheta'(t) \defeq		\begin{cases}
					\frac{t^{\vartheta-1}}{1-\vartheta},	&	\text{ for } t \leq e^{1/(1-\vartheta)},	\\	
					t^{\vartheta - 1} \log t,			&	\text{ for } t \geq e^{1/(1-\vartheta)}
					\end{cases}
\end{align*}
and $g_\vartheta(x) \defeq \int_0^x g_\vartheta'(t) \, \dt$, $x \geq 0$.
Again by Karamata's theorem, $g_\vartheta(x) \sim \frac1\vartheta x^{\vartheta} \log x$ as $x \to \infty$,
i.e., $g_\vartheta(x)$ is of the same order of growth as $f_1(x^\vartheta)$ as $x \to \infty$.
Similarly, $g_\vartheta(x) = \frac1\vartheta x^{\vartheta}$ and $f_1(x^\vartheta) = x^\vartheta$
for small $x$. Consequently, again we find a constant $C_\vartheta > 0$
such that $g_\vartheta(x) \leq f_1(x^\vartheta)$ for all $x \geq 0$.
On the other hand, as $g_\vartheta'$ is non-increasing, $g_\vartheta$ is subadditive and hence
\begin{align*}
\E[g_\vartheta(|U_t|)] \leq \E \bigg[\sum_{u \in \I_t} g_\vartheta \big(e^{- S(u)} |X_u|\big)\bigg]
\leq 2 C_\vartheta \E \bigg[\sum_{u \in \I_t} f_1\big(e^{- \vartheta S(u)} \big)f_1(|X_u|^\vartheta)\bigg].
\end{align*}
Again by Lemma \ref{lem:XlogX}, this is bounded for sufficiently small $t$.

\noindent
Now let $s,t \geq 0$. By conditioning with respect to $\F_{t \wedge s}$,
the $\sigma$-algebra containing all information up to and including time $t \wedge s$,
and using the Markov property, we infer
\begin{align*}
g(t,s)
&= \E[|U_t-U_s|^\vartheta]
\leq \E\bigg[\sum_{u \in \I_{t \wedge s}} e^{-\vartheta S(u)} g(|t-s|,0)  \bigg]	\\
&= m(t \wedge s,\vartheta) \cdot g(|t-s|,0) \to 0
\end{align*}
as $t$ is kept fixed and $s \to t$ by the first part of the proof.
\end{proof}

\begin{proof}[Proof of Lemma \ref{Lem:lattice to arbitrary convergence}]
Suppose that \eqref{eq:random sum scaled lattice} holds for all fixed $\delta > 0$.
Let $f: \R \to \R$ be differentiable with derivative $f'$ such that both $f$ and $f'$ are continuous and bounded.
Define, for $t \geq 0$,
\begin{equation*}
h(t) \defeq \E\big[f\big(t^{\frac3{2\vartheta}} e^{-F(\vartheta) t} U_t \big)\big].
\end{equation*}
By \eqref{eq:random sum scaled lattice}, we have
\begin{equation*}
h(n\delta) \to \E[f(Z_\delta)]
\quad	\text{as } n \to \infty,
\end{equation*}
for all $\delta > 0$.
If we can show that $h$ is continuous, then the Croft-Kingman lemma \cite[Theorem 2]{Kingman:1963} applies
and gives that $\E[f(Z_\delta)]$ is independent of $\delta$ and that $\lim_{t \to \infty} h(t) = \E[f(Z_1)]$.
Since the bounded continuously differentiable functions with bounded derivative are convergence determining on $\R$,
this implies \eqref{eq:random sum scaled limit}. 

\noindent
We now turn to the proof of the continuity of $h$.
For any $x,y \in \R$, we have
\begin{align*}
|f(x)-f(y)| \leq (\|f'\|_\infty \cdot |x-y|) \wedge (2 \|f\|_\infty)
\leq C |x-y|^{\vartheta\wedge 1}
\end{align*}
for some finite constant $C \geq 0$. Consequently, for any $s,t \geq 0$,
\begin{align*}
|h(s)-h(t)| \leq C \E\big[\big| s^{\frac3{2\vartheta}} e^{-F(\vartheta) s} U_s
-t^{\frac3{2\vartheta}} e^{-F(\vartheta) t} U_t \big|^{\vartheta\wedge 1}\big].
\end{align*}
The latter expression tends to $0$ as $s \to t$ by Lemma \ref{Lem:L^r-continuous}.

To prove the second part of the Lemma note that the process $U_t$ satisfies the following branching
relation
\begin{equation}	\label{eq:s2}
U_{t+s} \eqdist \sum_{u\in {\mathcal I_t}} e^{-S(u)} U_{s,u},
\end{equation} where $(U_{s,u})_u$ are independent copies of $U_s$, independent of the process up to time $t$.
Then \eqref{eq:s2} entails
\begin{equation*}
U_t \eqdist  {\1}_{\{E>t\}} X + {\1}_{\{E \leq t\}}\sum_{k=1}^N
A_k U_{t-E,k}
\end{equation*}
and therefore
\begin{align*}
t^{\frac3{2\vartheta}} e^{-F(\vartheta) t} U_t
&\eqdist  {\1}_{\{E>t\}} t^{\frac3{2\vartheta}} e^{-F(\vartheta) t} X
+  {\1}_{\{E \leq t\}} e^{-F(\vartheta)E} \sum_{k=1}^N
t^{\frac3{2\vartheta}}
e^{-F(\vartheta) (t-E)} \! A_k U_{t-E,k}.
\end{align*}
Now passing with $t$ to $\infty$ we conclude that $Z_1$ satisfies \eqref{eq:smoothing equation}.
\end{proof}

\section{Convergence along lattices}	\label{sec:lattice convergence}

Throughout the whole Section \ref{sec:lattice convergence},
we fix some $\delta>0$ and prove that \eqref{eq:random sum scaled lattice} holds
for a non-degenerate random variable $Z_\delta$.

\subsection{Properties of the skeleton branching random walk.}		\label{subsec:skeleton BRW}

The sequence of point processes $(\cZ_{n\delta})_{n \in \N_0}$
forms a discrete-time (or skeleton) branching random walk,
in which each individual produces offspring
with displacement relative to its position given by the points
of an independent copy of the point process $\cZ_\delta$.
In this section, we shall discuss the properties of this branching random walk that are relevant to us.

As $\delta$ is kept fixed throughout Section \ref{sec:lattice convergence},
we abbreviate $m(\delta,\theta)$, defined in \eqref{eq:m(t,theta)}, by $m(\theta)$.
For $n\in\N_0$ and $u \in \I_{n \delta}$, we define
\begin{equation}	\label{eq:V(u)}
V(u)	\defeq	\vartheta S(u)+n\log m(\vartheta) = \vartheta S(u) + n\delta \Phi(\vartheta).
\end{equation}
By the  definition of $\vartheta$, we have
\begin{equation}	\label{eq:boundary case}
\E\bigg[\sum_{u \in \I_\delta} e^{-V(u)}\bigg] = 1
\quad	\text{ and }	\quad
\E\bigg[\sum_{u \in \I_\delta} V(u)e^{-V(u)}\bigg] = 0,
\end{equation}
i.e., the branching random walk $(\sum_{u \in \I_{n\!\delta}} \delta_{V(u)})_{n \in \N_0}$
is in the \emph{boundary case}.\footnote{This notion was coined by Biggins and Kyprianou in \cite{Biggins+Kyprianou:2005}.}
Indeed, the first equation in \eqref{eq:boundary case} follows from \eqref{eq:m(t,theta)}.
Regarding the second, first notice that, by \eqref{eq:boundary case assumption} and \eqref{eq:Phi(theta)},
\begin{equation*}
\E\bigg[\sum_{|u|=1} \vartheta S(u) e^{- \vartheta S(u)}\bigg] = -\vartheta\E \bigg[\sum_{j \geq 1}A_j^\vartheta  \log A_j\bigg]
= -\E \bigg[\sum_{j \geq 1} A_j^\vartheta \bigg] + 1
= -\Phi(\vartheta)
\end{equation*}
and hence, by the many-to-one lemma (see e.g.\ \cite[Proposition 11]{Biggins+Kyprianou:2005}),
for every $n \in \N_0$,
\begin{align*}
\E\bigg[\sum_{|u|=n} \vartheta S(u) e^{- \vartheta S(u)}\bigg] = -n\Phi(\vartheta)(\Phi(\vartheta)+1)^{n-1}.
\end{align*}
Consequently,
\begin{align*}
\E\bigg[\sum_{u \in \I_\delta} & V(u) e^{-V(u)}\bigg]
= \E\bigg[\sum_{u \in \I_\delta} \vartheta S(u) e^{-\vartheta S(u)}\bigg] e^{-\delta \Phi(\vartheta)}
+ \delta \Phi(\vartheta) \E\bigg[\sum_{u \in \I_\delta} e^{-V(u)}\bigg]	\\
&=  e^{-\delta \Phi(\vartheta)} \sum_{n \geq 0} \E \bigg[\sum_{|u|=n} \!  \vartheta S(u) e^{-\vartheta S(u)}
\1_{\{\sigma(u) \leq \delta < \sigma(u)+E(u)\}} \bigg] + \delta \Phi(\vartheta)	\\
&= -\Phi(\vartheta) e^{-\delta \Phi(\vartheta)} \sum_{n \geq 1} n(\Phi(\vartheta)+1)^{n-1}
\Prob(T_n \leq \delta < T_{n+1}) + \delta \Phi(\vartheta)
= 0
\end{align*}
since $\Prob(T_n \leq \delta < T_{n+1}) = e^{-\delta} \frac{\delta^n}{n!}$ for all $n \in \N_0$.
Moreover, it is non-lattice by (A\ref{enum:non-lattice})
and satisfies
\begin{equation}	\label{eq:sigma^2}
\E\bigg[\sum_{u \in \I_\delta} V(u)^2 e^{-V(u)}\bigg] \in (0,\infty).
\end{equation}
The latter follows from Lemma 3.6 in \cite{Bogus+al:2019}.
As a corollary of Lemma \ref{lem:XlogX}, we get the following:
\begin{proposition}
Suppose that assumptions (A\ref{enum:existence vartetha}) and (A\ref{enum:XlogX condition}) are fulfilled. Then for the branching random walk defined by \eqref{eq:V(u)} it holds
\begin{align}	
\label{eq:moment condition 1}
&\E\bigg[\bigg(\sum_{u \in \I_\delta} e^{-V(u)}\bigg) \log_+^2\bigg(\sum_{u \in \I_\delta} e^{-V(u)}\bigg)\bigg] < \infty
\quad	\text{and}	\\
\label{eq:moment condition 2}
&\E\bigg[\bigg(\sum_{u \in \I_\delta} e^{-V(u)} V(u)_{+} \bigg)\log_+\bigg(\sum_{u \in \I_\delta} e^{-V(u)} V(u)_{+} \bigg)\bigg] < \infty.
\end{align}
\end{proposition}
\begin{proof}
An application of Lemma \ref{lem:XlogX} with $V=\vartheta S, p=2$
and $s=\delta$ gives \eqref{eq:moment condition 1}. For  \eqref{eq:moment condition 2} note that
\begin{equation*}\vartheta S(u)_+e^{-\vartheta S(u)}\le e^{-\tilde S(u)},\end{equation*}
where $\tilde S(u)\defeq\sum_{i=1}^{|u|}\vartheta(S(u_i)-S(u_{i-1}))-\log(1+\vartheta(S(u_i)-S(u_{i-1}))_+)$.
Hence the application   of Lemma \ref{lem:XlogX} with $V=\tilde S, p=1$ and $s=\delta$ finishes the proof.
\end{proof}

We have now checked that the assumptions of \cite[Theorem 1.1]{Madaule:2017} hold
and infer, with $V_{n}(u) \defeq V(u)-\frac{3}{2}\log(n)$ for $u\in \I_{n\delta}$,
\begin{equation*}
\cZ_n^\circ	\defeq	\sum_{u \in \I_{n\!\delta}} \delta_{V_n(u)}
\distto \cZ_{\infty}^\circ
\end{equation*}
where $\cZ_{\infty}^\circ$ is a point process on $\R$ satisfying $\cZ_{\infty}^\circ((-\infty,0]) < \infty$ a.\,s.
Here, the convergence in distribution is in the space of locally finite point measures
equipped with the topology of vague convergence.
For $k \in \N$, define $P_k \defeq \inf\{t \in \R: \cZ_{\infty}^\circ((-\infty,t]) \geq k\}$,
that is, $- \infty < P_1 \leq P_2 \leq P_3 \leq \ldots$ and $\cZ_{\infty}^\circ = \sum_{k \in \N} \delta_{P_k}$.
Then \cite[Formula (5.4)]{Iksanov+Kolesko+Meiners:2018} gives
\begin{align}	\label{eq:finite series}
\sum_{k\in\N} e^{-\beta P_{k}} < \infty	\quad	\Prob\text{-a.\,s.}
\end{align}
for every $\beta > 1$.
For simplicity of notation, suppose that $\cZ_{\infty}^\circ$ is defined on the probability space $(\Omega,\F,\Prob)$
and that it is independent of the families $(A(u),E(u))_{u \in \I}$ and $(X_u)_{u \in \I}$.
In particular, $(X_k)_{k \in \N}$ is independent of $\cZ_{\infty}^\circ$.
We consider the following random sums
\begin{equation*}
Z_n^* \defeq \sum_{k=1}^n e^{-\frac{P_k}{\vartheta}} X_k,		\quad	n \in \N.
\end{equation*}
Our main result, Theorem \ref{Thm:beyond the boundary}, follows directly from Lemma \ref{Lem:lattice to arbitrary convergence} and the following proposition.

\begin{proposition}	\label{Prop:extremal convergence skeleton}
Suppose that (A\ref{enum:non-lattice}) through (A\ref{enum:XlogX condition}) hold and $\mu_0\in\meas^1_r(\R)$ for some  $ r\in( \vartheta ,2]$.
Then $Z_n^* \Probto Z$ as $n \to \infty$ for some non-degenerate random variable $Z$ and
\begin{equation*}
m(\vartheta)^{-\frac{n}\vartheta} (n\delta)^{\frac3{2\vartheta}}\sum_{u \in \I_{n\!\delta}} e^{-S(u)} X_u \distto Z,
\end{equation*}
i.e., \eqref{eq:random sum scaled lattice} holds.
\end{proposition}

The bulk of the proof of this proposition can be adopted from the proof of Theorem 2.5 in
\cite{Iksanov+Kolesko+Meiners:2018},
however at some points changes are needed. In what follows, we repeat the major steps of the proof of
the cited theorem adjusted to the situation here and point out the changes that are required.

Define the following point processes on $\R^2$:
\begin{equation*}
\cZ_{\infty}^* \defeq \sum_{k \in \N} \delta_{(P_k,X_k)}
\quad	\text{and}	\quad
\cZ_{n}^* \defeq \sum_{|u|=n}\delta_{(V_n(u),X_u)},	\quad	n \in \N_0.
\end{equation*}

\begin{lemma}	\label{Lem:convergence of Z^*_n}
Suppose that the assumptions of Proposition \ref{Prop:extremal convergence skeleton} are satisfied.
Then $\int f(x,y) \, \cZ_{n}^*(\dx,\dy) \to \int f(x) \, \cZ_{\infty}^*(\dx,\dy)$ for all bounded continuous functions
$f: \R^2 \to \R$ satisfying $f(x,y) = 0$ for all sufficiently large $x$.
\end{lemma}
\begin{proof}[Source]
The lemma is a special case of Lemma 5.2 in \cite{Iksanov+Kolesko+Meiners:2018}.
\end{proof}

\begin{lemma}	\label{Lem:remainder}
Under the assumptions of Proposition \ref{Prop:extremal convergence skeleton},
for any $\delta>0$ and any measurable $h:\R\mapsto\R$ satisfying $0 \leq h_K \leq \1_{[K,\infty)}$,
we have
\begin{equation*}
\lim_{K\to\infty}\limsup_{n\to\infty}
\Prob\bigg(\bigg| \sum_{|u|=n}e^{-\frac1{\vartheta}V_n(u)}h_K(V_n(u))X_u  \bigg| > \delta \bigg) = 0.
\end{equation*}
\end{lemma}
\begin{proof}
The lemma follows from (the proof of) Lemma 5.3 in \cite{Iksanov+Kolesko+Meiners:2018},
except at one point in the proof where the Topchi\u\i--Vatutin inequality is used (Lemma A.1 in the cited reference).
The use of the latter inequality has to be replaced by an application of \eqref{eq:moment subadditivity}
The rest of the proof carries over without changes.
\end{proof}

We are now ready to prove Proposition \ref{Prop:extremal convergence skeleton}.

\begin{proof}[Proof of Proposition \ref{Prop:extremal convergence skeleton}]
Recall that $0 < \vartheta < r \leq 2$.
Let $\beta_0 \defeq \frac r\vartheta > 1$.
Given $\cZ_\infty^\circ$, for each $n\in\N$,
the random variable $Z_n^*$ is a sum of independent random variables
(centered in the case $\vartheta \geq 1$).
Then, for any $\delta > 0$ and any $n,m \in \N$ with $m \leq n$, by \eqref{eq:moment subadditivity},
\begin{align*}
\Prob(| Z_n^*-Z_m^*|>\delta | \cZ_\infty^\circ )
\leq \delta^{-r}\E[|X|^r] \cdot
\sum_{k=m+1}^n  e^{-\beta_0 P_k}
\end{align*}
and the second term converges to zero as $m,n \to \infty$ by \eqref{eq:finite series}.
Hence, conditionally given $\cZ_\infty^\circ$,
$(Z_n^*)_{n \in \N_0}$ forms a Cauchy sequence in probability and thus converges in probability.
We denote the limit in probability of the sequence $(Z_n^*)_{n \in \N_0}$
by $Z$.

The proof of the second part is based on the decomposition
\begin{align*}
(m(\vartheta))^{-\frac n \vartheta} n^{\frac3{2\vartheta}}\sum_{u \in \I_{n\!\delta}}e^{-S(u)} X_u
&= \sum_{u \in \I_{n\!\delta}}e^{-\frac1\vartheta V_n(u)} X_u	\\
&= \sum_{u \in \I_{n\!\delta}}e^{-\frac1\vartheta V_n(u)} f_K(V_n(u)) X_u	\\
&\hphantom{=} +\sum_{u \in \I_{n\!\delta}}e^{-\frac1\vartheta V_n(u)} (1-f_K(V_n(u))) X_u	\\
&~\eqdef Z_{n,K}+R_{n,K},
\end{align*}
where $f_K$ is a continuous function such that $\1_{(-\infty,K]} \leq f_K \leq \1_{(-\infty,K+1]}$.
The remainder of the proof is based on an application of Theorem 4.2 in \cite{Billingsley:1968}.
In view of Lemma \ref{Lem:remainder}, the cited theorem gives the assertion
once we have shown the following two assertions:

\begin{enumerate}[1.]
	\item	$Z_{n,K} \distto Z_K^*$ as $n \to \infty$ for every fixed $K>0$ where $Z_K^*$ is some finite random variable;
	\item	$Z_K^* \Probto Z$ as $K \to \infty$.
\end{enumerate}
The first assertion is a consequence of Lemma \ref{Lem:convergence of Z^*_n}.
Indeed, the function on $\R^2$ that maps $(x,y)$ to $e^{-\frac{1}{\vartheta} x} f_K(x) y$
is continuous and vanishes for all sufficiently large $x$.
Therefore, Lemma \ref{Lem:convergence of Z^*_n} yields
\begin{align*}
Z_{n,K} &= \sum_{u \in \I_{n\!\delta}} e^{-\frac{1}{\vartheta} V_n(u)} f_K(V_n(u)) X_u	\\
&=		\int e^{-\frac{1}{\vartheta} x} f_K(x) y \, \cZ_{n}^*(\dx,\dy)
\distto	\int e^{-\frac{1}{\vartheta} x} f_K(x) y \, \cZ_\infty^*(\dx,\dy)
\eqdef	Z_K^*.
\end{align*}
The second assertion can be proved similarly as in the proof of Theorem 2.5 in \cite{Iksanov+Kolesko+Meiners:2018}.
More precisely, it follows from the dominated convergence theorem once we have proved that
\begin{equation*}
\Prob(|Z-Z_K^*| > \varepsilon \mid \cZ_\infty^\circ) \to	 0	\quad	\text{a.\,s.}
\end{equation*}
as $K \to \infty$ for every $\varepsilon > 0$. Now fix $\varepsilon > 0$ and observe that
\begin{align*}
\Prob(|Z-Z_K^*| > \varepsilon \mid \cZ_\infty^\circ)
&\leq \Prob(\{|Z_n^*-Z_K^*| > \varepsilon \text{ for infinitely many } n\} \mid \cZ_\infty^\circ)	\\
&= \E\big[ \liminf_{n \to \infty} \1_{\{|Z_n^*-Z_K^*| > \varepsilon\}} \mid \cZ_\infty^\circ \big]	\\
&\leq \liminf_{n \to \infty} \Prob(|Z_n^*-Z_K^*| > \varepsilon \mid \cZ_\infty^\circ)	\\
&\leq \liminf_{n \to \infty} \varepsilon^{-r} \E[|Z_n^*-Z_K^*|^r \mid \cZ_\infty^\circ],
\end{align*}
where Fatou's lemma gives the first inequality and Markov's inequality the second.
Now given a realization $P_1 \leq P_2 \leq \ldots $ of the point process $\cZ_\infty^\circ$,
we can choose $n \in \N$ such that $P_n > K+1$.
Then
\begin{align*}
\E[|Z_n^*-Z_K^*|^r \mid \cZ_\infty^\circ]
&= \E\bigg[\bigg|\sum_{k=1}^n e^{-\frac{P_k}{\vartheta}} (1-f_K(P_k)) X_k\bigg|^r \, \bigg| \, \cZ_\infty^\circ  \bigg]	\\
&\leq 2 \sum_{k=1}^n e^{-\frac{r}{\vartheta} P_k} (1-f_K(P_k))^r \, \E[|X_k|^r],
\end{align*}
where we have used inequality \eqref{eq:moment subadditivity}.
The latter can be estimated as follows
\begin{align*}
\sum_{k=1}^n e^{-\frac{p}{\vartheta} P_k} (1-f_K(P_k))^r \, \E[|X_k|^r]
\leq \E[|X_1|^r] \sum_{k \geq 1: P_k > K} e^{-\frac{r}{\vartheta} P_k}
\to 0	\quad	\text{a.\,s.}
\end{align*}
as $K \to \infty$ by \eqref{eq:finite series} since $\frac r\vartheta > 1$.
\end{proof}

\subsubsection*{Acknowledgements.}
D.\,B.~and K.\,K.~were partially supported by the National Science Center, Poland (Sonata Bis, grant number DEC-2014/14/E/ST1/00588).
The research of M.\,M.\ was supported by DFG Grant ME 3625/3-1.
The work was initiated while D.\,B.\ was visiting Innsbruck in May 2019.
He gratefully acknowledges hospitality and the financial support again by DFG Grant ME 3625/3-1.

%
%
%
%

\bibliographystyle{plain}
\bibliography{self}

\end{document}